\documentclass[12pt]{amsart}
\usepackage[english]{babel}
\usepackage{amssymb, amsthm, amsmath}
\usepackage{txfonts}
\usepackage{calrsfs}
\usepackage{mathrsfs}
\usepackage{amscd}
\usepackage{graphicx}
\usepackage[applemac]{inputenc} 
\usepackage[T1]{fontenc} 
\usepackage{hyperref} 
\usepackage{breakurl}
\usepackage{cite}

\DeclareMathOperator*{\colim}{colim} 
\DeclareMathOperator*{\holim}{holim}

\DeclareMathOperator*{\hocolim}{hocolim}

\DeclareMathOperator{\Sing}{\textrm{Sing}_*}

\begin{document}
\pretolerance=4000 

\bibliographystyle{../hsiam}

\setcounter{tocdepth}{2} 
\setlength{\parindent} {0pt}

\newcommand{\cat}[1]{\mathscr{#1}} 
\newcommand{\ob}{\textrm{Ob}}
\newcommand{\mor}{\textrm{Mor}}
\newcommand{\id}{\mathbf 1} 

\newcommand{\sSet}{\mathbf{sSet}}
\newcommand{\Alg}{\mathbf{Alg}}
\newcommand{\sPr}{\mathbf{sPr}}
\newcommand{\dgCat}{\mathbf{dgCat}} 
\newcommand{\dgAlg}{\mathbf{dgAlg}} 
\newcommand{\Vect}{\mathbf{Vect}} 
\newcommand{\Ch}{\mathbf{Ch}}
\newcommand{\Chd}{{\mathbf{Ch}_{dg}}} 
\newcommand{\Chdp}{{\mathbf{Ch}_{pe}}} 
\newcommand{\Chp}{{\mathbf{Ch}_{pe}}} 
\newcommand{\uChp}{{\underline {\mathbf{Ch}}{}_{pe}}} 
\newcommand{\skMod}{\mathbf{skMod}} 
\newcommand{\Cat}{\mathbf{Cat}} 
\newcommand{\sCat}{\mathbf{sCat}}
\newcommand{\sModCat}{\mathbf{sModCat}}

\newcommand*\cocolon{
        \nobreak
        \mskip6mu plus1mu
        \mathpunct{}
        \nonscript
        \mkern-\thinmuskip
        {:}
        \mskip2mu
        \relax
}

\newcommand{\Hom}{\mbox{Hom}}
\newcommand{\uHom}{\mbox{\underline{Hom}}}
\newcommand{\Aut}{\mbox{Aut}}
\newcommand{\Out}{\mbox{Out}}
\newcommand{\End}{\mbox{End}}
\newcommand{\mods}{\textrm{-Mod}}
\newcommand{\Map}{\mbox{Map}}
\newcommand{\map}{\mbox{map}}
\newcommand{\Tor}{\mbox{Tor}}
\newcommand{\Ext}{\mbox{Ext}}
\newcommand{\Kos}{\mbox{Kos}} 
\newcommand{\Spec}{\mbox{Spec}}

\newcommand{\set}[1]{\mathbb{#1}}
\newcommand{\Q}{\mathbb{Q}}
\newcommand{\C}{\mathbb{C}}
\newcommand{\Z}{\mathbb{Z}}
\newcommand{\R}{\mathbb{R}}

\newcommand{\De}{\Delta}
\newcommand{\Ga}{\Gamma}
\newcommand{\Om}{\Omega}
\newcommand{\ep}{\epsilon}
\newcommand{\de}{\delta}
\newcommand{\La}{\Lambda}
\newcommand{\la}{\lambda}
\newcommand{\al}{\alpha}
\newcommand{\om}{\omega}

\renewcommand{\to}{\rightarrow}
\newcommand{\we}{\tilde \to} 
\newcommand{\cof}{\hookrightarrow}
\newcommand{\fib}{\twoheadrightarrow}
\newcommand{\acof}{\tilde \hookrightarrow}
\newcommand{\afib}{\tilde \twoheadrightarrow}

\newcommand{\oo}{\infty}
\newcommand{\di}{\mbox{d}} 

\newcommand{\op}{^{\textrm{op}}} 
\newcommand{\inv}{^{-1}} 
\newcommand{\nneg}{\tau_{\geq 0}} 

\newcommand{\Bold}{\boldsymbol}
\newcommand{\IF}{\textrm{if }}
\newcommand{\Res}{\textrm{Res}}
\newcommand{\comment}[1]{}

\theoremstyle{plain}
\newtheorem{thm}{Theorem}
\newtheorem*{thm*}{Theorem}
\newtheorem{cor}[thm]{Corollary}
\newtheorem{lemma}[thm]{Lemma}
\newtheorem{propn}[thm]{Proposition}
\newtheorem{conj}{Conjecture}

\makeatletter
\newtheorem*{rep@theorem}{\rep@title}
\newcommand{\newreptheorem}[2]{
\newenvironment{rep#1}[1]{
 \def\rep@title{#2 \ref{##1}}
 \begin{rep@theorem}}
 {\end{rep@theorem}}}
\makeatother

\newtheorem{theorem}{Theorem}
\newreptheorem{theorem}{Theorem}

\theoremstyle{definition}
\newtheorem*{defn}{Definition}
\newtheorem*{altdef}{Alternative Definition}
\newtheorem{eg}{Example}
\newtheorem*{conv}{Convention}
\newtheorem*{notation}{Notation}
\newtheorem*{fact}{Construction}
\newtheorem{qn}{Question}
\newtheorem{fqn}{Further Question}

\theoremstyle{remark}
\newtheorem{rk}{Remark}
\newtheorem*{claim}{Claim}

\title{Morita Cohomology}
\author{Julian V. S. Holstein}
\address{Christ's College and University of Cambridge}
\email{J.V.S.Holstein@dpmms.cam.ac.uk}

\begin{abstract}
We consider two categorifications of the cohomology of a topological space $X$ by taking coefficients in the category of differential graded categories. 
We consider both derived global sections of a constant presheaf and singular cohomology and find the resulting dg-categories are quasi-equivalent and moreover quasi-equivalent to representations in perfect complexes of chains on the loop space of $X$.
\end{abstract}

\maketitle

\section{Introduction}

In this paper we categorify the cohomology of topological spaces by considering coefficients in the category of differential graded categories. 

We begin with the calculation of derived global sections $R\Ga(X, \underline k)$, for $k$ a field and $X$ a locally contractible space.
By definition these are derived global sections of a constant sheaf. We categorify by considering the constant presheaf $\underline k$ not as a presheaf of rings, but as a presheaf of dg-categories with one object, where we equip dg-categories with the Morita model structure. In this model structure $k \simeq \Chp$, which is fundamental to our construction. 

Hence we call categorified cohomology \emph{Morita cohomology}. We write $\cat H^{M}(X)$ for the dg-category $R\Ga(X, \underline k)$ computed with this model structure. 

The following characterization as categorified \v Cech cohomology follows once we establish a local model structure on presheaves of dg-categories. 
\begin{reptheorem}{thm-cohomology-holim}
Given a good hypercover 
$\{U_i\}_{i \in I}$ of $X$ one can compute $\cat H^M(X) \simeq \holim_{i \in I\op} \Chp$. 
\end{reptheorem}

To categorify singular cohomology we recall the action of simplicial sets on dg-categories, $(K, \cat D) \mapsto \cat D^{K}$. If we fix the second variable this construction is well-known to give a Quillen adjunction from $\sSet$ to $\dgCat$.
Then for a  
topological space $X$ one defines categorified singular cohomology as $\cat Y(X) \coloneqq \Chp^{\Sing X}$. Here $\Chp$ denotes perfect chain complexes over an arbitrary commutative ring. We call $\cat Y(X)$ the category of  \emph{$\oo$-local systems} and we prove the following comparison theorem:
\begin{reptheorem}{thm-monodromy}
The category $\cat H^{M}(X)$ is quasi-equivalent to the category of $\oo$-local systems $\cat Y(X)$. 
\end{reptheorem}

Homotopy invariance and a Mayer--Vietoris theorem are easy to establish for $\oo$-local systems and hence for Morita cohomology.

The category $\cat Y(K)$ is closely related to the based loop space of $X$, as is shown by the next result. Here we denote by $\Chp^{C_{*}(\Om X)}$ the category of representations of the dg-algebra $C_{*}(\Om X)$ of chains on the space of (Moore) loops which have perfect underlying complex. 
\begin{reptheorem}{thm-loopreps}
If $X$ is a pointed and connected topological space, the category $\cat Y(X)$ is quasi-equivalent to $\Chp^{C_{*}(\Om X)}$.
\end{reptheorem}

We then establish a method of computing $\Chp^{C_{*}(\Om X)}$ if $X$ is a CW complex.

One can compute the Hochschild homology of Morita cohomology in several cases. For example for a simply connected space $HH_*(\cat H^M(X)) \simeq H^*(\cat L X)$, where the right hand side is cohomology of the free loop space. This follows from results available in the literature. 

Let us also mention here that a very explicit description of Morita cohomology is proved in the companion paper \cite{HolsteinB}:
\begin{thm*}
Let $X$ be a CW complex. 
Then the dg-category $\cat H^{M}(X)$ is quasi-equivalent to 
the dg-category of homotopy locally constant sheaves of perfect complexes.
\end{thm*}

All of these results are from the author's thesis.

\subsection*{Relation to other work}
We collect some ideas and results from the literature which are related to the constructions here and in \cite{HolsteinB}. This is not meant to be an exhaustive list.

Our results can be considered as a version of derived or higher non-abelian cohomology for topological spaces.  
Cohomology with higher categorical coefficients is considered for example in \cite{Simpson96a} where Simpson constructs the higher non-abelian cohomology stack of the de Rham stack of a smooth projective variety as a certain internal hom-space in geometric stacks.

Considering a simplicial set $K$ as a constant stack To\"en and Vezzosi construct a derived mapping stack $\Map(K, R\textrm{\underline{Perf}})$, where $R\textrm{\underline{Perf}}$ is the moduli stack of perfect complexes. This appears for example in \cite{Pantev11}.  
The construction of $\oo$-local systems in \mbox{Section \ref{ch-ils}} is a non-geometric version of this, which is already somewhat interesting and more tractable then the mapping stack. 

Morita cohomology is also closely related to the schematic homotopy type of \cite{Katzarkov09}.
In fact, Morita cohomology is equivalent to the category of perfect complexes on the schematic homotopy type, as follows readily from the characterization as homotopy locally constant sheaves. (This was pointed out to me by Jon Pridham.) For a different view of the schematic homotopy type see \cite{Pridham13}.

There is an analogue of the main theorem of \cite{HolsteinB} in the coherent setting:
Under suitable conditions global sections of the presheaf of dg-categories associated to the structure sheaf of a scheme can be computed as the category of perfect complexes of coherent sheaves. This appears for example as Theorem 2.8 in \cite{Simpson05} referring back to \cite{Simpson01}. 

The equivalence between $\oo$-local systems and homotopy locally constant sheaves that is obtained by combining this paper with\cite{HolsteinB} is a linear and stable version of results in \cite{Toen08b} or \cite{Shulman08}, where the corresponding result for presheaves of simplicial sets is proved by going via the category of fibrations. Another view on locally constant functors is given in \cite{Cisinski09}.

\subsection*{Outline} 

After briefly recalling some technical results and definitions in Section \ref{sect-prelim} we define a local model structure on presheaves of dg-categories and define their cohomology in \ref{sect-cohomodel}. We also characterize fibrant presheaves of dg-categories as hypersheaves. We then explicitly hypersheafify the constant presheaf in \ref{sect-sheafification} and use this to define Morita cohomology $\cat H^{M}(X)$ and compute it as the homotopy limit of a constant diagram with fiber $\Chp$ indexed by the distinct open sets of a good hypercover.

In Section \ref{ch-ils} we define a category of $\oo$-local systems from the cotensor action of simplicial sets on dg-categories, and show it es equivalent to $\cat H^{M}(X)$ in \ref{sect-infinity}.  We use this to identify $\cat H^{M}(X)$ with representations of chains on the loop space, in \ref{sect-loop-rep}. 
Section \ref{sect-cellular} is then concerned with providing an explicit method for computing the category of $\oo$-local systems. In \ref{sect-hochschild} we collect some results about finiteness of $\cat H^{M}(X)$ and show how to compute Hochschild (co)homology in some cases.

We conclude by computing some explicit examples in Section \ref{ch-eg}.

\subsection*{Acknowledgements}

I would like to thank Ian Grojnowski for many insightful comments and stimulating discussions.

\section{Preliminaries}\label{sect-prelim}

\subsection{Notation and conventions}
We assume the reader is familiar with the theory of model categories, but will try to recall all the less well-known facts about them that we use.

In any model category we write $Q$ for functorial cofibrant replacement and $R$ for functorial fibrant replacement.
We write mapping spaces (with values in $\sSet$) in a model category $\cat M$ as $\Map_{\cat M}(X,Y)$. 
All other enriched hom-spaces in a category $\cat D$ will be denoted as $\uHom_{\cat D}(X,Y)$. 
In particular we use this notation for differential graded hom-spaces, internal hom-spaces and hom-spaces of diagrams enriched over the target category. It should always be clear from context which category we enrich in. 

We will work over the (underived) commutative ground ring $k$. 
We assume characteristic 0 in order to freely use differential graded constructions. Most of Section \ref{ch-morita} will moreover assume $k$ is a field.

$\Ch = \Ch_k$ will denote the model category of chain complexes 
over the ring $k$ equipped with the projective model structure where fibrations are the surjections and weak equivalences are the quasi-isomorphisms.
Note that we are using homological grading convention, i.e.\ the differential decreases degree. 
We write $\Chp$ for the subcategory of compact objects in $\Ch$, these are exactly the perfect complexes. 

$\Chd$ denotes the dg-category whose object are fibrant and cofibrant objects of $\Ch$. 
Note that there is a natural identification of the subcategory of compact objects in $\Chd$ with $\Chp$.

\subsection{Differential graded categories}
We recall some important definitions and results about dg-categories. Basic references are \cite{Keller06} and \cite{Toen07a}. Many technical results are proved in \cite{Tabuada07}.

Let $\dgCat$ denote the category of categories enriched in $\Ch$. Given $\cat D \in \dgCat$ we define the homotopy category $H_{0}(\cat D)$ as the category with the same objects as $\cat D$ and $\Hom_{H_{0}(\cat D)}( A, B) = H_{0}\uHom_{\cat D}(A, B)$.
If $\cat D$ is a model category enriched in $\Ch$ we define $L\cat D$ as its subcategory of fibrant cofibrant objects. 
We say $\cat D$ is a dg-model category if the two structures are compatible, that is if they satisfy the pushout-product axiom, see for example the definitions in Section 3.1 of \cite{Toen07a}. Then $Ho(\cat D) \simeq H_{0}(L\cat D)$, where we take the homotopy category in the sense of model categories on the left and in the sense of dg-categories on the right.

Recall that there are two model structures on $\dgCat$.
Firstly there is the \emph{Dwyer--Kan model structure}, denoted $\dgCat_{DK}$.

Recall the functor $\cat D \mapsto \cat D\mods$ sending a dg-category to its model category of modules, i.e.\ $\cat D\mods$ is the category of functors \mbox{$\cat D \to \Ch$} and strict natural transformations. This is naturally a cofibrantly generated model category enriched in $\Ch$
whose fibrations and weak equivalences are given levelwise. We usually consider its subcategory of fibrant and cofibrant objects, $L(\cat D\mods)$. 

\begin{rk}\label{rk-mod-model}
The construction of the model category $\cat D\mods$ follows Chapter 11 of \cite{Hirschhorn03}, but there are some changes since we are considering enriched diagrams. We sketch the argument here for lack of a reference. Let $I$ and $J$ denote the generating cofibrations and generating trivial cofibrations of $\Ch$. 
The generating (trivial) cofibrations of $\cat D\mods$ are then of the form $h^{X} \otimes A \to h^{X} \otimes B$ for $A \to B \in I$ (resp. $J$), where $h^X$ denotes the contravariant Yoneda embedding. As in Theorem 11.6.1 of \cite{Hirschhorn03} we transfer the model structure from $\Ch^{\textrm{discrete}(\cat D)}$. This works since $h^{X}$ is compact in $\cat D\mods$ and so are its tensor products with the domains of $I$, ensuring condition (1) of Theorem 11.3.2 holds. For the second condition we have to check that relative $J \otimes h^{X}$-cell complexes are weak equivalences. Pushouts are constructed levelwise. The generating trivial cofibrations of $\Ch$ are of the form $0 \to D(n)$. Since the pushout $U \leftarrow 0 \to D(n)$ is weakly equivalent to $U$  we are done. 

\end{rk}

\begin{rk}
In order to satisfy the smallness assumption we will always assume that all our dg-categories are small relative to some larger universe.
\end{rk}

The homotopy category of the model category $L\cat D\op\mods$ is called the derived category of $\cat D$ and denoted $D(\cat D)$.

\begin{defn} We denote by $\dgCat_{Mor}$ the category of dg-categories with the \emph{Morita model structure}, i.e.\ the Bousfield localization of $\dgCat_{DK}$ along functors that induce equivalences of the derived categories, see Chapter 2 of \cite{Tabuada07}. 
\end{defn}

Fibrant objects in $\dgCat_{Mor}$ are dg-categories $\cat A$ such that the homotopy category of $\cat A$ is equivalent (via Yoneda) to the subcategory of compact objects of $D(\cat A)$ \cite{Keller06}. We can phrase this as: every compact object is quasi-representable. An object $X \in D(\cat A)$ is called \emph{compact} if $\Hom_{D(\cat A)}(X, -)$ commutes with arbitrary coproducts. We denote by $()_{pe}$ the subcategory of compact objects.
Morita fibrant dg-categories are also called \emph{triangulated} since their homotopy category is an (idempotent complete) triangulated category. 

With these definitions $\cat D \mapsto L(\cat D\op \mods)_{pe}$, often denoted the \emph{triangulated hull}, is a fibrant replacement, for example $k\mapsto \Chp$.

The category $\dgCat$ is symmetric monoidal with tensor product $\cat D \otimes \cat E$ given as follows. The objects are $\ob \cat D \times \ob \cat E$ and $\uHom_{\cat D \otimes \cat E}((D, E), (D', E')) \coloneqq \uHom_{\cat D}(D, D') \otimes \uHom_{\cat E}(E, E')$. The unit is the one object category $k$, which is cofibrant in either model structure.

While $\dgCat$ is not a monoidal model category there is a derived internal Hom space and the mapping spaces in $\dgCat_{Mor}$ can be computed as follows \cite{Toen06}: Let $R\uHom(\cat C, \cat D)$ be the dg-category of right-quasirepresentable $\cat C \otimes \cat D\op$-modules, i.e.\ functors \mbox{$F\colon \cat C \otimes \cat D\op \to \Ch$} such that for any $c \in \cat C$ 
we have that $F(c, -)$ is isomorphic in the homotopy category to a representable object in $\cat D\op\mods$ and moreover cofibrant. Then $R\uHom$ is right adjoint to the derived tensor \mbox{product $\otimes^{L}$}.
Moreover $\Map(\cat C, \cat D)$ is weakly equivalent to the nerve of the subcategory of quasi-equivalences in $R\uHom(\cat C, \cat D)$. We will quote further properties of this construction as needed.

We will need the following lemma relating the two model structures. It follows since $\dgCat_{Mor}$ is a left Bousfield localization, hence the derived functor of the identity, given by fibrant replacement, preserves homotopy limits.

\begin{lemma}\label{lemma-dgcatholim}
Fibrant replacement as a functor from $\dgCat_{Mor}$ to $\dgCat_{DK}$ preserves homotopy limits.
\end{lemma}

This means we can compute homotopy limits in $\dgCat_{Mor}$ by computing the homotopy limit of a levelwise Morita-fibrant replacement in $\dgCat_{DK}$.

We will abuse notation and write $R$ for the dg-algebra $R$ as well as for the 1-object dg-category with endomorphism space $R$ concentrated in degree 0. 

Recall that there is a model structure on differential graded algebras over $k$ with unbounded underlying chain complexes, which can be considered as the subcategory of one-object-categories in $\dgCat_{DK}$.

\begin{rk}
While we are working with differential graded categories we are facing some technical difficulties for lack of good internal hom-spaces. It would be interesting to know if another model of stable linear $(\infty, 1)$-categories, e.g. \cite{Lurie11}, could simplify our treatment.
\end{rk}

\subsection{Simplicial resolutions}\label{sect-framings}

Model categories are naturally models for $\oo$-categories and in fact have a notion of mapping spaces. Even if a model category is not enriched in $\sSet$ one can define mapping spaces in $Ho(\sSet)$.
One way to do this is by defining simplicial resolutions, which we will make extended use of. Since this construction is crucial to our results we recall it here.

Let $\De$ be the simplex category and consider the constant diagram functor $c\colon \cat M \to \cat M^{\De\op}$. Then a \emph{simplicial resolution} $M_{*}$ for $M \in \cat M$ is a fibrant replacement for $cM$ in the Reedy model structure on $\cat M^{\De\op}$. (For a definition of the Reedy model structure see for example Chapter 15 of \cite{Hirschhorn03}.)

For example, this construction allows one to compute mapping spaces: If $cB \to \tilde B$ is a simplicial resolution in $\cat M^{\De\op}$ and $QA$ a cofibrant replacement in $\cat M$ then $\Map(A, B) \simeq \Hom^{*}(QA, \tilde B) \simeq R(\Hom^{*}(-, c-))$, where the right-hand side uses the bifunctor $\Hom^{*}\colon \cat M\op \times \cat M^{\De\op} \to \mathbf{Set}^{\De\op}$ that is defined levelwise.
The dual notion is a \emph{cosimplicial resolution}.

Recall that $\cat V$ is a \emph{symmetric monoidal model category} if it is both symmetric monoidal and a model category and the structures are compatible, to be precise they satisfy the pushout-product axiom, see Definition 4.2.1 in \cite{Hovey07}. 
This means in particular that tensor and internal Hom give rise to Quillen functors. We then call the adjunction of two variables satisfying the pushout-product axiom a \emph{Quillen adjunction of two variables}.

Similarly a \emph{model $\cat V$-category} $\cat M$ is a model category $\cat M$ that is tensored, cotensored and enriched over $\cat V$
such that the pushout product axiom holds. We call a model $\Ch$-category a \emph{dg-model category}.

For example a model $\sSet$-category, better known as a \emph{simplicial model category}, $\cat M$ consists of the data $(\cat M, \Map, \otimes, \map)$ where the enrichment
$\Map\colon \cat M\op \times \cat M \to \sSet$, the cotensor (or power)
$\map\colon \sSet\op \times \cat M \to \cat M$ 
and the tensor $\otimes\colon \sSet \times \cat M \to M$ satisfy the obvious adjointness properties (in other words, they form an adjunction of two variables). Here the pushout-product axiom says that the natural map $f \Box g\colon {A \otimes L \amalg_{A \otimes K} B \otimes K} \to B \otimes L$ is a cofibration if $f$ and $g$ are and is acyclic if $f$ or $g$ is moreover acyclic.

While not every model category is simplicial, every homotopy category of a model category is enriched, tensored and cotensored in $Ho(\sSet)$.
In fact, $\cat M$ can be turned into a simplicial category in the sense that there is an enrichment $\Map$ and there are a tensor and cotensor which can be constructed from the simplicial and cosimplicial resolutions.
Let a cosimplicial resolution $A^* \in \cat M^\De$ and a simplicial set $K$ be given. Consider $\De K$, the category of simplices of $K$, with the natural map $u\colon \De K \to \De$ sending $\De[n] \mapsto K$ to $[n]$.
We define $A^* \otimes K = \colim_{\De K} A^{n}$ to be the image of $A^*$ under $\colim \circ \ u^*\colon \cat C^\De \to \cat C^{\De K} \to \cat C$. 
Similarly there is $A^K \coloneqq \map(K, A)$ which is the image of the simplicial resolution $A_* \in \cat M^{\De \op}$ under $\lim \circ \ v^*$, where $v\colon \De K\op \to \De\op$. This can also be written as $A^K = \lim_n (\prod_{K_n} A_{n}$).

If $\cat M$ is a simplicial category one can use $(RA)^{\De^{n}}$ for $A_{n}$ and $(QA) \otimes \De^{n}$ for $A^{n}$. 

\begin{rk}\label{rk-cotensor-holim}
Note that $A^{K}$ can also be written as a homotopy limit, $\holim_{\De K\op} A_{n}$. This follows for example from Theorem 19.9.1 of \cite{Hirschhorn03}, the conditions are satisfied by Propositions 15.10.4 and  16.3.12.
\end{rk}

The functor $(A, K) \mapsto A^K$ is adjoint to the mapping space construction $A, B \mapsto \Hom(QB, A_*) \in \sSet$. Similarly $(B, K) \mapsto B \otimes K$ is adjoint to the mapping space construction $A, B \mapsto \Hom(B^*, RA) \in \sSet$, see Theorem 16.4.2 in \cite{Hirschhorn03}.
Hence on the level of homotopy categories the two bifunctors together with $\Map$ give rise to an adjunction of two variables. This is of course not a Quillen adjunction, but it is sensitive enough to the model structure to allow for certain derived functors. 
We will quote further results about this construction as needed.

\subsection{Homotopy ends}\label{sect-holim}

We will freely use homotopy limits in model categories, see \cite{Hirschhorn03} for background. Since they are less well-treated in the literature we recall the construction of \emph{homotopy ends} of bifunctors.
Recall that an end is a particular kind of limit. Let $\al(I)$ denote the \emph{twisted arrow category} 
of $I$: Objects are arrows, $f\colon i \to j$, and morphisms are opposites of factorizations, i.e.\ $(f\colon i \to j) \Rightarrow (g\colon i' \to j')$ consists of maps $i' \to i$ and $j \to j'$ such that their obvious composition with $f$ equals $g$. Then there are natural maps $s$ and $t$ (for source and target) from $\al(I)$ to $I\op$ and $I$ respectively. For a bifunctor $F\colon I\op \times I \to \cat C$ one defines the end $\int_i F(i, i)$ to be $\lim_{\al(I)} (s \times t)^*F$. 
Then the homotopy end is: \[\int^h_i F(i,i) \coloneqq \holim_{\al(I)} (s \times t)^* F\]
Details on this view on homotopy ends can be found (dually) in \cite{Isaacson09}. 

The canonical example for an end is that natural transformations from $F$ to $G$ can be computed as $\int_A \Hom(FA, GA)$. A similar example of the use of homotopy ends is provided by the computation of mapping spaces in the diagram category of a model category $\Map(A_{\bullet}, B_{\bullet}) \simeq \int^{h}_{i} \Map(A_{i}, B_{i})$. The case of simplicial sets is dealt with in \cite{Dwyer83}.

In general, we have the following lemma. Assume $\cat M^{I}$ exists with the injective model structure  
and let $Q$ and $R$ denote cofibrant and fibrant replacement in this model category.

\begin{lemma}\label{lemma-htpy-end}
Consider a right Quillen functor $\underline H\colon \cat M\op \times \cat M \to \cat V$. 
Then there is a natural Quillen functor $(F, G) \to \int_{i} \underline H(Fi,Gi)$ from $(\cat M^{I})\op \times \cat M^I$ to $\cat V$ whose derived functor is 
\[(F, G) \mapsto \int_{i}\underline H(QFi, RGi)\]
which is weakly equivalent to
\[(F, G) \mapsto \int^{h}_{i} R\underline H(Fi, Gi)\]
\end{lemma}
\begin{proof}
The $\cat V$-structure exists by standard results in \cite{Kelly82}.
It is in fact a model $\cat V$-category. One can check the pushout-product axiom levelwise; this is enough as cofibrations are defined levelwise.
Hence the derived functor is $(F, G) \mapsto \int_{i}\underline H(QFi, RGi)$. 

On the other hand $\int_{i} \underline H(Fi,Gi)$ is the composition of levelwise hom-spaces with the limit, 
\[\lim \circ \ \underline H^{\al(I)} \circ (s\times t)^{*}\colon (\cat M^{I})\op \times \cat M^{I} \to (\cat M\op \times \cat M)^{\al(I)} \to \cat V^{\al(I)} \to \cat V\]
But then the derived functor is the composition of derived functors, $\int^{h}_{i} R\underline H(Fi, Gi)$. 

This is a little subtle, since our aim is to avoid fibrantly replacing at the level of diagram categories. However, the levelwise derived functor $R\underline H$ from $(\cat M\op)^{I} \times \cat M^{I}$ to $\cat V^{\al(I)}$ is a derived functor. This is the case since levelwise fibrant replacement gives a right deformation retract in the sense of 40.1 in \cite{Dwyer04} since $(s\times t)^{*}$ preserves all weak equivalences and levelwise $\underline H$ preserves weak equivalences between levelwise fibrant objects. 
\end{proof}

\begin{rk}
A slight modification of the lemma implies
the formula for mapping spaces. We just have to replace $\underline H$ by $\Hom^{*}\colon \cat M\op \times \cat M^{\De\op} \to \sSet$ and adjust the proof accordingly.
\end{rk}

\section{Morita Cohomology}\label{ch-morita}

\subsection{Cohomology of presheaves of model categories}\label{sect-cohomodel}

In this section we define what we mean by cohomology of a presheaf with coefficients in a model category. We also prove two technical results that are needed for the computation of Morita cohomology in the next section. 

Let us assume $\cat M$ is a cellular and left proper model category. 
The case we are interested in is $\cat M = \dgCat_{Mor}$ over a field $k$.
It follows from \cite{HolsteinC} that $\dgCat$ with the Morita or the Dwyer-Kan model structure is cellular and is left proper if the ground ring is a field. 

We will consider the category $\cat M^J$ of presheaves on a category $J\op$ with values in a model category $\cat M$. If $\cat M$ is left proper and cellular then so is $\cat M^{J}$.
We will denote by $\cat M^J_{proj}$ the projective model structure on $\cat M^{J}$ with levelwise weak equivalences and fibrations, and whose cofibrations are defined by the lifting property. If $\cat M$ is cofibrantly generated this is well-known to be a model structure, which is cellular and left proper if $\cat M^{J}$ is.

We are interested in enriching the model category $\cat M^{J}_{proj}$.
Let us start by recalling the case where the construction is straightforward. Let $\cat V$ be a monoidal model category and assume that it has a cofibrant unit. $\cat V = \sSet, \Ch$ are examples. Then if $\cat M$ is a model $\cat V$-category, 
then it is easy to check that so is $\cat M^J$. In particular if $\cat M$ is monoidal then $\cat M^J$ is a model $\cat M$-category. 
We can write $\uHom$ for the enriched hom-spaces, and the functor $\Hom\colon (\cat M^{J})\op \times \cat M^{J} \to \cat M$ is right Quillen and there is a derived functor $R\uHom$ obtained by fibrant and cofibrant replacement, see Lemma \ref{lemma-htpy-end}

If $\cat M$ is monoidal and a model category, but not a monoidal model category, then we can still construct an $\cat M$-enrichment of $\cat M^J_{proj}$ as a plain category, which will of course not be a model category 
enrichment. We define $\uHom_{M^J}(A, B) = \int_j \uHom(A(j), B(j))$, see \cite{Kelly82}.

Note that this enrichment is not in general derivable, i.e.\ weak equivalences between cofibrant and fibrant pairs of objects do not necessarily go to weak equivalences. So defining a suitable substitute for $R\uHom$ takes some care, see the proof of Lemma \ref{lemma-fibrancy}.

We have to consider this case since our example of interest is $\cat M = \dgCat$, which a symmetric monoidal category and a model category, see \cite{Toen07a}, but not a symmetric monoidal model category. (The tensor product of two cofibrant objects need not be cofibrant.)

Now fix a locally contractible topological space $X$, for example a CW complex, and consider presheaves on $Op(X)$. We consider the Grothendieck topology induced by the usual topology on $X$ and write the site as $(\mathbf{Set}^{Op(X)\op}, \tau)$. In other words $\tau$ is just the collection of maps represented by open covers. 
(We will not use any more general Grothendieck topologies or sites.)
We let $J = Op(X)\op$. Our aim is to localize presheaves on $Op(X)$ with respect to covers in $\tau$.

Recall that a \emph{left Bousfield localization} of a 
 model category $\cat N$ at a set of maps $S$ is a left Quillen functor $\cat N \to \cat N_S$ that is initial among left Quillen functors sending the elements of $S$ to isomorphisms in the homotopy category.
We need to know that left Bousfield localizations of $\cat M^{I}$ exist.

\begin{lemma}
Assume $\cat N$ is a cellular and left proper model category and let $S$ be a set of maps. Then $\cat N_{S}$ exists. 
The cofibrations are equal to projective cofibrations, 
weak equivalences between are $S$-local weak equivalences and fibrant objects are $S$-local objects.
\end{lemma}
\begin{proof}
This is Theorem 4.1.1 of \cite{Hirschhorn03}. 
\end{proof}
Recall for future reference that an object $P$ is $S$-local if it is fibrant in $\cat N$ and every $f\colon A \to B \in S$ induces a weak equivalence $\Map_{\cat N}(B, P) \simeq \Map_{\cat N}(A, P)$. A map $g\colon C \to D$ is an $S$-local weak equivalence if it induces a weak equivalence $\Map_{\cat N}(D, P) \simeq \Map_{\cat N}(C, P)$ for every $S$-local $P$.

Given a set $N$ we write $N \cdot M \coloneqq \amalg_N M \in \cat M$ for the tensor over $\mathbf{Set}$ and extend this notation to presheaves.

\begin{defn}
Let $\cat M^J_{\tau} \coloneqq (\cat M^{J}_{proj})_{H_{\tau}}$ denote the left Bousfield localization of $\cat M^J_{proj}$ with respect to 
\[H_{\tau} = \{S \cdot \id_{\cat M} \to h_{W} \cdot \id_{\cat M}
\ | \ S \to h_{W} \in \tau\}\]
Here $h_{-}$ denotes the covariant Yoneda embedding $X \mapsto \Hom(-, X)$.
\end{defn}

We have assumed $\cat M$ and hence $\cat M^{J}$ is cellular and left proper. Since $H_{\tau}$ is a set the localization $\cat M^{J}_{\tau}$ exists.

We have now localized with respect to \v Cech covers. We are interested in the \emph{local model structure} which is obtained by localizing at all hypercovers. 

\begin{rk}
By way of motivation see \cite{Dugger04a} for the reasons that localizing at hypercovers gives the local model structure on simplicial presheaves, i.e.\ weak equivalences are precisely stalk-wise weak equivalences.
\end{rk}

\begin{defn}
A \emph{hypercover} of an open set $W \subset X$ is a simplicial presheaf $U_*$ on the topological space $W$ such that:
\begin{enumerate}
\item For all $n \geq 0$ the presheaf $U_n$ is isomorphic to a disjoint union of a small family of presheaves representable by open subsets of $W$. We can write $U_n = \amalg_{i \in I_n} h_{U_n^{(i)}}$ for a set $I_n$ where the $U_n^{(i)} \subset W$ are open.
\item The map $U_0 \to *$ is in $\tau$, 
i.e.\ the $U_0^{(i)}$ form an open cover of $W$.
\item For every $n \geq 0$ the map $U_{n+1} \to (cosk_n U_*)_{n+1}$ is in $\tau$. Here $(cosk_n U)_{n+1} = M_n^WU $ is the $n$-th matching object computed in simplicial presheaves over $W$. 
\end{enumerate}

\end{defn}

Intuitively,  the spaces occurring in $U_{1}$ form a cover for the intersections of the $U^{(i)}_{0}$, the spaces in $U_{2}$ form a cover for the triple intersections of the $U_{1}^{(i)}$ etc. To every \v Cech cover one naturally associates a hypercover in which all $U_{n+1} \to (cosk_{n} U_{*})_{n+1}$ are isomorphisms.

Note that despite the notation $U_n$ is not an open set but a presheaf on open sets that is a coproduct of representables.

We denote by $I = \cup I_{n}$ the category indexing the representables making up the hypercover. 
Associated to any hypercover of a topological space is the simplicial space $n \mapsto \amalg_{i \in I_{n}} U_{n}^{i}$ which is also sometimes called a hypercover.

Hypercovers are naturally simplicial presheaves. We work with presheaves with values in a more general model category. The obvious way to associate to a simplicial object in a model category a plain object is to take the homotopy colimit.

\begin{defn}\label{defn-hypercover}
Let the set of \emph{hypercovers in $\cat M^J$} be defined as
\[\check H_{\tau} = \{\hocolim_I (U_* \cdot \id_{\cat M} )\to h_{W} \cdot \id_{\cat M}
\ | \ U_* \to h_{W} \textrm{ a hypercover\}}\] 
where 
we take the levelwise tensor and the homotopy colimit in $\cat M^J$ with the projective model structure. Since disjoint union commutes with cofibrant replacement we could equivalently take the limit of $U_{n}$ over $\De\op$, the opposite of the simplex category. 
\end{defn}

\begin{rk} 
Note that the homotopy colimit does not change if instead we use the localised model structure $\cat M^J_\tau$. 
\end{rk}

\begin{defn}
Let the left Bousfield localization of $\cat M^{J}_{\tau}$ at the hypercovers $\check H_{\tau}$ be denoted by $\cat M^J_{\check \tau}$ and call it the \emph{local model structure}.
\end{defn}

The localization exists just as before. The fibrant objects are the $\check H_{\tau}$-local objects of $\cat M^{J}$.

Note that $\uHom(h_W, \cat F) \simeq \cat F(W)$ if the model structure on $\cat M^J$ is enriched over $\cat M$. So we sometimes write hypercovers as if they are open sets. For example given a hypercover $U_*$ and a presheaf $\cat F \in \cat M^J$ we write $\cat F(U_n)$ for $\uHom(U_n, \cat F)$ etc. 
In particular $\cat F(U_{n}) = \cat F(\amalg_i U^{(i)}_n) \coloneqq \prod_i \cat F(U^{(i)}_n)$.

To compute cohomology we need to compute the derived functor of global sections. First we need to know that pushforward is right Quillen. 
The arguments in the proofs of the following two lemmas are Propositions 1.22 and 3.37 in \cite{Barwick07}, we repeat them for the reader's convenience.
\begin{lemma}
Consider a map $r\colon C \to D$ of diagrams and a model category $M$. Then there is a Quillen adjunction $r_!\colon M^C_{proj} \rightleftarrows M^D_{proj}  \cocolon r^*$. 
\end{lemma}
\begin{proof} We define $r^*$ by precomposition. Then $r_!$ exists as a Kan extension. Clearly $r^*$ preserves levelwise weak equivalences and fibrations.
\end{proof}
\begin{lemma}
Given any map $r: (C, \tau) \to  (D, \sigma)$ 
that preserves covers and hypercovers we get a Quillen adjunction 
$r_!\colon M^C_{\check \tau} \rightleftarrows M^D_{\check \sigma} \cocolon r^*$. The same adjunction works if we only localize with respect to \v Cech covers.
\end{lemma}
\begin{proof} To prove the result for the localization with respect to covers we use the universal property of localization applied to the map $M^C \to M^D \to M^D_{\sigma}$ which is left Quillen and sends hypercovers to weak equivalences and hence must factor through $M^C \to M^C_{\tau}$ in the category of left Quillen functors, giving rise to $r_{!} \vdash r^{*}$.

To prove the result for the localization at hypercovers we repeat the same argument for $M^C_{\tau} \to M^C_{\check \tau}$ etc.
\end{proof}

Consider now locally contractible topological spaces $X$ and $Y$ with sites of open sets $(Op(X),\tau)$ and $(Op(Y), \sigma)$.
Given a map $f\colon X \to Y$ consider $f\inv\colon (Op(X), \tau) \to (Op(Y), \sigma)$. Then $f_* \coloneqq (f\inv)^*$ and by the above it is a right Quillen functor. 

As usual we write $\Ga$ or $\Ga(X, -)$ for $(\pi_X)_*$ where $\pi_X \colon X \to *$.

\begin{defn} Let $\cat C$ be a presheaf with values in a model category $\cat M$ and let $\cat C^{\#}$ be a fibrant replacement for $\cat C$ in the local model category $\cat M^J_{\check \tau}$ defined above. Then we define global sections as
\[R\Ga(X, \cat C) = \cat C^{\#}(X)\]
\end{defn}

In Section \ref{sect-sheafification}
we will compute $\cat C^{\#}$ if $\cat C$ is constant.

Since a hypersheaf satisfies $\cat F(X) = \holim_i \cat F(U_i)$ for some cover $\{U_i\}$ we can also think of global section as a suitable homotopy limit. A concise formulation of this will be given in Theorem \ref{thm-cohomology-holim}. 

\begin{defn} Consider the presheaf $\underline k$ that is constant with value $k \in \dgCat$ and let $\underline k^{\#}$ be a fibrant replacement for $\underline k$. Then we define \emph{Morita cohomology} as
\[R\Ga(X, \underline k) \coloneqq R\Ga(X, \uChp) = \underline k^{\#}(X)\]
in $Ho(\dgCat_{Mor})$. We write $\cat H^M(X) \coloneqq R\Ga(X, \underline k)$.
\end{defn}

One can also consider the version with unbounded fibers, $R\Ga(X, \underline \Ch)$.
\begin{rk} As usual $R\Ga(\varnothing, \underline k) \simeq 0$, the terminal object of $\dgCat$.
\end{rk}
\begin{rk} The term cohomology is slightly misleading as our construction corresponds to the underlying complex and not the cohomology groups. It is an interesting question whether there is an analogue to taking cohomology and how it relates for example to semi-orthogonal decomposition as defined in \cite{Bondal02}. 
\end{rk}

We finish this section with two lemmas on the fibrant objects and weak equivalences in the local model structure.
\begin{defn}
We call a presheaf $\cat F$ a \emph{hypersheaf} 
if it satisfies 
\begin{equation}\label{sheafcondition}
\cat F(W) \simeq \holim_{I\op} \cat F(U_*) \textrm{ for every hypercover } U_* \textrm{ of every open } W \subset X
\end{equation}
The limit is over $I\op = \cup I_{n}$; we could write it $\holim_{n} \holim_{i \in I_{n}} \cat F(U_n^{(i)})$ which can be considered as $\holim_{n \in \De} \cat F(U_n)$ using the convention above. 
This condition is also called \emph{descent} with respect to hypercovers.
\end{defn}

For the next Lemma we need $\cat M$ to have a certain homotopy enrichment over itself. For simplicity we specialise to $\cat M = \dgCat_{Mor}$.

\begin{lemma}\label{lemma-fibrancy}
Levelwise fibrant hypersheaves are fibrant in the above model structure.\end{lemma}
\begin{proof}
We need to show that for a levelwise fibrant presheaf $\cat F$ the hypersheaf condition on $\cat F$ implies that $\cat F$ is $\check H_{\tau}$-local, i.e.\ that whenever \mbox{$\epsilon\colon \hocolim (U_* \cdot \id) \to h_W \cdot \id$} is in $\check H_{\tau}$ then $\Map(\hocolim( U_* \cdot \id), \cat F) \simeq {\Map(h_W \cdot \id, \cat F)}$. 
We will show that both sides are weakly equivalent to $\Map_{\dgCat_{Mor}}(\id, \cat F(W))$. 

We need a suitable derived hom-space between hypersheaves of dg-categories with values in dg-categories. We define $R\uHom'(A_\bullet, B_\bullet) \coloneqq \int^h_V R\uHom(A_V, B_V)$, where $R\uHom$ is To\"en's internal derived Hom of dg-categories.

First note that 
\[ R\uHom'(h_W \cdot \id, \cat F) \simeq  \int^{h}_{V \subset W} R\uHom(\id, \cat F(V)) \simeq \holim_{V \subset W} \cat F(V)\]
The first weak equivalence holds since $h_W(V)$ is just the indicator function for $V \subset W$ and the second since the homotopy end over a bifunctor that is constant in the first variable degenerates to a homotopy limit, by comparing the diagrams. 
Then we observe $\holim_{V \subset W} \cat F(V) \simeq \cat F(W)$ if $\cat F$ satisfies the hypersheaf condition. 

We claim that this implies $\Map(h_{W} \cdot \id, \cat F) \simeq \Map(\id, \cat F(W))$. 
Note that in $\dgCat$ we have $\Map(A, B) \simeq \Map(\id, R\uHom(A, B))$, see Corollary 6.4 in \cite{Toen06}.
Moreover the mapping space in diagram categories is given by a homotopy end, see Lemma \ref{lemma-htpy-end}.

Putting these together we see $\Map(A_\bullet, B_\bullet) = \int^h_V \Map(\id, R\uHom(A_V, B_V))$.
Then the claim follows since $\Map(\id, -)$ commutes with homotopy limits and hence homotopy ends.

Similarly 
we have 
\begin{align*}
\Map(\hocolim_{i} (U_{i} \cdot \id), \cat F) & \simeq \holim_{i} \Map(\id, R\uHom'(U_i \cdot \id, \cat F)) \\
& \simeq \Map(\id, \holim_{i} \holim_{V \subset U_i} \cat F(V))
\end{align*}
which is $\Map(\id, \cat F(W))$ again by applying the hypersheaf condition twice.
\end{proof}

\begin{rk}\label{rk-enrichment}
The
theory of enriched Bousfield localizations from \cite{Barwick07} says that in the right setting $\cat M^J_{\check \tau}$ is an enriched model category and fibrant objects are precisely levelwise fibrant hypersheaves.
However, this theory requires that we work with a category $\cat M$ that is tractable, left proper and a symmetric monoidal model category with cofibrant unit. The characterization of fibrant objects in particular depends on the enriched hom-space being a Quillen bifunctor.
While $\dgCat_{Mor}$ is left proper and equivalent to a combinatorial and tractable subcategory,  cf. \cite{HolsteinC}, 
it is well-known $\dgCat_{Mor}$ is not symmetric monoidal. Tabuada's equivalent category $Lp$ of localizing pairs has a derivable internal Hom object, but is not a monoidal model category either. In fact, tensor product with a cofibrant object is not left Quillen. Consider the dg-category $\cat S(0)$ that is the linearization of $a \to b$.
The example $\cat S(0) \otimes \cat S(0)$ in $\dgCat$ gives rise to 
\[(\emptyset \subset \cat S(0)) \otimes (\emptyset \subset \cat S(0)) \simeq (\emptyset \subset \cat S(0) \otimes \cat S(0))\]
which is again a tensor product of cofibrant objects that is not cofibrant. Then $\uHom(\cat S(0), -)$ cannot be Quillen either.
\end{rk}

\begin{lemma}\label{lemma-hypercover-we}
Let $\cat M = \dgCat$. Assume that for two presheaves $\cat F$ and $\cat F'$ there is a hypercover $V_*$ on which $\cat F$ and $\cat F'$ agree and which restricts to a hypercover of $W$ for every open $W$. Then $\cat F$ and $\cat F'$ are weakly equivalent in $\cat M^J_{\check \tau}$.
\end{lemma}
\begin{proof}
We need to show that there is a $\check H_{\tau}$-local equivalence between $\cat F$ and $\cat F'$, i.e.\ $\Map_{\cat M^{J}}(\cat F, \cat G) \simeq \Map_{\cat M^{J}}(\cat F', \cat G)$ for any fibrant $G$.

Specifically, we consider sets $V$ in the hypercover of agreement contained in $W$. Then we know $\Map(\cat F(V), \cat G(V)) \simeq \Map(\cat F'(V), \cat G(V))$. 
To compute $\Map(\cat F, \cat G) = \int^h_W\Map(\cat F(W), \cat G(W))$ note that the homotopy end can be computed as follows:
\[\int^h_W\Map(\cat F(W), \cat G(W)) \simeq \int \Hom((Q^* \cat F)(W), R\cat G(W))\]
Here we use fibrant replacement and a cosimplicial resolution in $\cat M^J$.
But now $\holim_V \cat G(V) = \lim_V R\cat G(V)$ by fibrancy of the diagram $R\cat G$. So it suffices to consider
$\int_W \Hom(Q^* F(W), R\cat G(W))$ where $R\cat G(W) = \lim_{V \subset W} R\cat G(V))$. But an end is just given by the collection of all compatible maps, and every map from $Q^i\cat F(W)$ to $R\cat G(W)$ is determined by the maps from $Q^i\cat F(W)$ to $R\cat G(V)$, which
factor through $Q^i\cat F(V)$. So the end over the $V$ is the same as the end over all $W$ and
\begin{align*}
\Map(\cat F, \cat G) &\simeq \int_V \Hom(Q^* \cat F(V), R\cat G(V)) \\
&\simeq \int_V \Hom(Q^* \cat F'(V), R\cat G(V)) \simeq \Map(\cat F', \cat G)
\end{align*}
This completes the proof.
\end{proof}
\begin{rk}
If $\cat M$ is a symmetric monoidal model category then by Remark \ref{rk-enrichment} fibrant objects are precisely levelwise fibrant hypersheaves and are again determined on a hypercover and Lemma \ref{lemma-hypercover-we} holds again.
\end{rk}

\subsection{Sheafification of constant presheaves}\label{sect-sheafification}

Our aim now is to compute a hypersheafification of the constant presheaf with values in a model category.

Recall that $X$ is a locally contractible topological space and that we have fixed a model category $\cat M$ that is cellular and left proper. We now assume that $\cat M$ is moreover homotopy enriched over itself and has a cofibrant unit. We will also need that the derived internal hom-space commutes with homotopy colimits. 

The example we care about is $\cat M = \dgCat_{Mor}$.
The fact that $\holim R\uHom(A_{i}, B) \simeq R\uHom(\hocolim A_{i}, B)$ in $\dgCat$ follows from Corollary 6.5 of \cite{Toen06}. The one object dg-category $k$ is a cofibrant unit. 

We write $\underline P$ for the constant presheaf with fiber $P \in \cat M$.

First we need to quote two lemmas about comparing homotopy limits.
Given a functor $\iota\colon I \to J$, recall the natural map $e_j\colon (j \downarrow \iota) \to J$ from the undercategory, sending $(i, j \rightarrow \iota(i))$ to $\iota(i)$.
\begin{lemma}\label{lemma-duggersthm-2}
Let $\iota\colon I \to J$ be a
functor between small categories such that for every $j \in J$ the overcategory $(\iota \downarrow j)$ is nonempty with a contractible nerve and let $X\colon J \to \cat M$ be a diagram. Then the map $\holim_{J} X \to \holim_{I} \iota^{*} X$ is a weak equivalence.
\end{lemma}
\begin{lemma}\label{lemma-duggersthm}
Let $\iota\colon I \to J$ be a functor between small categories and let $X\colon J \to \cat M$ a diagram with values in a model category. Suppose that the composition
\[X_j \to \lim_{(j \downarrow \iota)} e_j^*(X) \to \holim_{(j \downarrow \iota)} e_j^*(X)\]
is a weak equivalence for every $j$. Then the natural map $\holim_J X \to \holim_I \iota^*X$ is a weak equivalence.
\end{lemma}
\begin{proof} [Proofs ]For topological spaces these are Theorems 6.12 and 6.14 of \cite{Dugger} and the proofs 
do not depend on the choice of model category.
\end{proof}

We will also rely on the following results from \cite{Dugger04b}. The first statement is Theorem 1.3 and the second is a corollary of Proposition 4.6 as any basis is a complete open cover.

\begin{propn}\label{propn-hypercover-colim}
Consider a hypercover $U_* \to X$ of a topological space as a simplicial space. Then the maps $\hocolim U_* \to |U_* | \to X$ are weak equivalences in $\mathbf{Top}$.
\end{propn}
The colimit here is over the category $\De\op$, but recall that $\hocolim_{\De\op} U_{n} \simeq \hocolim_{I} U_{n}^{i}$. 
\begin{propn}\label{propn-completecover-colim}
Consider a basis $\mathfrak U$ of a topological space $X$ as a simplicial space. Then the map $\hocolim_{U \in \mathfrak U} U \to X$ is a weak equivalence in $\mathbf{Top}$. 
\end{propn}

Let $X$ be locally contractible. Then we can define the (nonempty) set $\{\mathfrak U^s\}_{s\in S}$ of all bases of contractible sets for $X$. 

\begin{defn}\label{defn-sheafification}
Fix a basis of contractible sets $\mathfrak U^s$ for $X$.
Let $\underline P$ be a constant presheaf with fiber $P \in \cat M$ and define a presheaf $\cat L^s_{\underline P}$ 
by 
\[
\cat L^s_{\underline P}(U) = 
\holim_{V \subset U, V \in \mathfrak U^{s}} \underline {RP}(V)\]
where $P \to RP$ is a fibrant replacement in $\cat M$. Denote the natural map by $\la\colon \underline P \to \cat L^s_{\underline P}$. The restriction maps are induced by inclusion of diagrams.
\end{defn}

We will be interested in $ \mathcal L^{s}_{\underline k} \simeq \mathcal L^s_\uChp$.

This construction proceeds via constructing rather large limits, so even the value of $\cat L^s$ on a contractible set is hard to make explicit.

The following lemma is the first step towards showing that our construction does indeed give a hypersheaf.
\begin{lemma}\label{lemma-contractibles}
Consider a constant presheaf $\underline P$ 
 with fibrant fiber $P \in \cat M$ on $Op(X)$.  
 Then on any contractible set $U \subset Op(X)$ we have $\cat L^s_{\underline P}(U) \simeq P$.
\end{lemma}
\begin{proof}

Consider $\mathfrak U$ as a category. 
We need to show that $\holim_{\mathfrak U\op} \underline P \simeq P$.
The crucial input is that the weak equivalences $V \to *$ give rise to $U \simeq \hocolim_{V \subset \mathfrak U} V \simeq \hocolim_{\mathfrak U} *$ via Proposition \ref{propn-completecover-colim}.

Now consider any $N \in \cat M$ and a cosimplicial resolution $N^{*}$. 
Then we have the functor $K \mapsto N \otimes K$ defined in the introduction which is left Quillen, as is shown in Corollary 5.4.4 of \cite{Hovey07}.
Hence it preserves homotopy colimits and we have: 
\[N = N \otimes \hocolim_{\mathfrak U} * \simeq \hocolim_{\mathfrak U}  (N \otimes *) = \hocolim_{\mathfrak U}  \underline N\]

Finally, we use the fact that $\cat M$ has internal hom-spaces. Replace $N$ above be the cofibrant unit. Then we conclude: 
\begin{align*}
\holim_{U \in {\mathfrak U\op}} \underline P(U) & \simeq \holim_{{\mathfrak U\op}} R\uHom(\id, \underline P(U)) \\
& \simeq \holim_{{\mathfrak U\op}} R\uHom(\underline \id(U), P) \simeq R\uHom(\hocolim_{\mathfrak U}  \underline \id, P) \\
& \simeq R\uHom(\id, P) \simeq P
\end{align*}
In the second line we use the fact that $R\uHom(-, P)$ sends homotopy colimits to homotopy limits.
\end{proof}

\begin{propn}\label{propn-sheaf-independent}
For two choices $\mathfrak U^{t}$ and $\mathfrak U^s$ there is a chain of quasi-isomorphisms between $\cat L^{t}_{\underline P}$ and $ \cat L^{s}_{\underline P}$. Hence there is a presheaf $\cat L_{\underline P}$ well defined in the homotopy category.
\end{propn}
\begin{proof}
By considering the union of of $\mathfrak U^{s}$ and $\mathfrak U^{t}$ 
it suffices to show the result if $\mathfrak U^{t}$ is a subcover of $\mathfrak U^{s}$. By Lemma \ref{lemma-duggersthm} it then suffices to fix $U_{i} \in \mathfrak U^{t}$ and check that $\holim_{i/\iota} \underline P \simeq P$ where $\iota$ is the natural inclusion map.
But the arrow category stands for the opposite of the category of all the elements of $\mathfrak U^{s}$ contained in $U_i$. These form a basis and hence the homotopy limit is given be Lemma \ref{lemma-contractibles}.
\end{proof}

\begin{propn}\label{propn-havesheaf}
For any choice of $\mathfrak U^{s}$ the presheaf $\cat L^s_{\underline P}$ is fibrant, i.e.\ it is $\check H$-local. 
\end{propn}
\begin{proof}
By Lemma \ref{lemma-fibrancy}, it is enough to show $\cat L_{\underline P}^s$ is levelwise fibrant (immediate from definition) and satisfies the hypersheaf condition.

Given
a hypercover $\{W_i\}_{i \in I}$ of $U$ we may assume that any element of $\mathfrak U^{s}$ is a subset of one of the $W_{i}$.
Then we consider for every $i$ the basis of contractibles $\mathfrak U^{s(i)}$ for $W_{i}$ of elements of $\mathfrak U^{s}$ that are contained in $W_{i}$.
We obtain the following:
\[\holim_{I\op} \cat L^s_{\underline P}(W_i) \simeq \holim_{i \in I\op} \holim_{U \in \mathfrak U^{s(i)}{}\op} \underline P(U) \leftarrow \holim_{U \in \mathfrak U^{s}{}\op} \underline P(U)\]
And our aim is to show the arrow on the right is a weak equivalence.

By considering $R\uHom(\id \otimes \hocolim *, P)$ as in the proof of Lemma \ref{lemma-contractibles} 
it suffices to show $\hocolim_{i \in I} \hocolim_{V \in \mathfrak U^{s(i)}} V \to \hocolim_{U \in \mathfrak U^{s}}U$ is a weak equivalence. But if we apply Proposition \ref{propn-completecover-colim} this is weakly equivalent to $\hocolim_{i \in I} W_{i} \to X$, 
which is a weak equivalence by Proposition \ref{propn-hypercover-colim}.
\end{proof}

If $X$ is locally contractible 
then it has a basis of contractible open sets. Moreover one can associate a hypercover to any basis. For details on the construction see 
Section 4 of \cite{Dugger04b} and note that a basis is a complete cover. 

\begin{propn}\label{propn-lp-we}
If $\cat P$ is 
constant then the natural map $\cat P \to \cat L_{\cat P}$ is a weak equivalence of presheaves. 
\end{propn}
\begin{proof} To show that $\cat L$ resolves $\cat P$ it is enough to observe that $\cat L_{\cat P}(U) \simeq \cat P$ for contractible $U$ by Lemma \ref{lemma-contractibles}. Now the contractible opens give rise to a hypercover on which $\cat P$ and $\cat L_{\cat P}$ agree and that restricts to a hypercover on every open set. By Lemma \ref{lemma-hypercover-we} that suffices to prove the proposition.
\end{proof}

With Proposition \ref{propn-lp-we} we can compute $R\Ga(X, \underline P)$ as $\cat L_{\underline P}(X)$. Note that since we have not used functorial factorization this is not a functor on the level of model categories but only on the level of homotopy categories.

\begin{defn}
We will call a cover in $(\mathbf{Set}^{Op(X)\op},\tau)$ a \emph{good cover} if all its elements and all their finite intersections are contractible. Correspondingly a \emph{good hypercover} is a hypercover such that all its open sets $U_{n}^{(i)}$ are contractible. 
\end{defn} 

We will now consider a good hypercover $\{U_i\}_{i \in I}$. 
For computations it is easier not to consider the full simplicial presheaf given by open sets in the cover but only the semi-simplicial diagram of nondegenerate open sets, obtained by leaving out identity inclusions. 

\begin{thm}\label{thm-cohomology-holim}
Let $U_{*} \to h_X$ be a good hypercover of a topological space $X$. Let $\underline P$ be a constant presheaf on $X$. 
Then $R\Gamma(X, \underline P) \simeq \holim_{I\op} P \simeq \holim_{I_{0}\op} P$ where $I_{0}$ indexes the distinct contractible sets of $U_{*}$. 
\end{thm}
\begin{proof}
We consider a fibrant replacement $\cat L_{\underline P}$ as in Definition \ref{defn-sheafification}. Let $I$ index the connected open sets of $U_{*}$.  
Then we have:
\begin{align*}
R\Ga(X, \underline P) &\simeq \cat L_{ \underline P}(X) 
\simeq \holim \cat L_{\underline P}(U_*) \\
& \simeq \holim_{I\op} \cat L_{\underline P}(U_n^{(i)})  \simeq \holim_{I\op} P
\end{align*}
Here we use Lemma \ref{lemma-contractibles} to identify $\cat L_{\underline P}(U_n^{(i)})$ and $P$.
Now consider $\iota\colon I_{0}\op \subset I\op$ and note that all the overcategories $\iota \downarrow i$ are trivial (any $i \in I$ is isomorphic to some $j \in I_{0}$) so by Lemma \ref{lemma-duggersthm-2} we have
\[R\Ga(X, \underline U) \simeq \holim_{I_{0}\op} P \qedhere\]
\end{proof}

\begin{rk}
Note that we can of course take the hypercover associated to a \v Cech cover in this theorem.
In fact, since we are concerned with locally constant presheaves we could also compute in the \v Cech model structure, but considering hypercovers simplifies the theory. 
\end{rk}

We conclude this section with some results on functoriality.

\begin{lemma}\label{lemma-Leray-Serre}
Let $f\colon X \to Y$ be a continuous map and let $\underline P_{X \textrm{ or } Y}$ denote the constant presheaf with fiber $P$ on $X$ or $Y$. 
Then $R\Ga(X, \underline P) \simeq R\Ga(Y, Rf_*(\underline P))$.
\end{lemma}
\begin{proof}
The fact that $R\Ga \circ Rf_* \simeq R\Ga$ follows immediately from $\pi_{Y,*} \circ f_* = \pi_{X,*}$ and the fact that all these maps preserve fibrations.
\end{proof}

\begin{lemma}
With notation as above
there is a functor from $R\Ga(Y, \underline P_Y)$ to $R\Ga(X, \underline P_X)$.
\end{lemma}
\begin{proof}
$\Ga$ is a covariant functor. From Lemma \ref{lemma-Leray-Serre} we have a natural weak equivalence $R\Ga(Y, Rf_*(\underline P_X))\to R\Ga(X, \underline P_X)$.

Let $\underline P_{\bullet} \to \cat P_{\bullet}$ be a fibrant replacement. It is then enough to construct a map $f^\bullet\colon \cat P_Y \to Rf_*(\cat P_X)$ of hypersheaves on $Y$. On any open set $U$ this is given by $\underline P_Y(U) = P \to f_* \underline P_X(U) \to f_* \cat P_X(U)$. \end{proof}

\begin{rk}
With $P = \Chp$ this gives functoriality for Morita cohomology if we use functorial factorizations.

Note that our computation using good covers is not functorial unless we pick compatible covers. 
However, if $X$ and $Y$ have bases of contractible sets which are suitably compatible there is a natural comparison map between homotopy limits. 
\end{rk}

\begin{rk} The results of the last sections relied on the assumption that $\dgCat$ is left proper, which is only the case if $k$ is of flat dimension zero.

Nevertheless, one can consider the question of what Morita cohomology should be over other ground rings and it is sensible to use $\Ga(X, \cat L_{\uChp})$ 
as our \emph{definition} of Morita cohomology if $k$ has positive flat dimension. All pertinent results then still apply, in particular Theorem \ref{thm-cohomology-holim}, and we can prove equivalence with the category of $\oo$-local systems in Section \ref{sect-infinity}. 
\end{rk}

\section{Infinity-local systems}\label{ch-ils}

\subsection{Singular cohomology with coefficients in $\dgCat$}\label{sect-infinity}

We will now consider the categorification of singular cohomology, given by the dg-category of $\oo$-local systems. Here we consider dg-categories over an arbitrary commutative ring $k$.

Recall from \ref{sect-framings} that while the model categories on $\dgCat$ are not simplicial,
there is a bifunctor $\sSet\op \times \dgCat_{DK} \to \dgCat_{DK}$ that induces a
natural $Ho(\sSet)$ cotensor action on $Ho(\dgCat_{DK \textrm{ or } Mor})$. We write this as $(K, \cat D) \mapsto \cat D^K$.

\begin{defn}
We define the \emph{dg-category of $\oo$-local systems} on a simplicial set $K$ as $\Chp^K$. We write $\cat Y(K)$ for $\Chp^K$. For a topological space $X$ we recall the (unpointed) singular simplicial set $\Sing(X)$ and define $\cat Y(X) \coloneqq \cat Y(\Sing(X))$. We also define $\cat Y^{u}(K) = \Chd^{K}$ and $\cat Y^{u}(X) = \Chd^{\Sing(X)}$.
\end{defn}
\begin{rk}
We are using the Dwyer--Kan model structure for simplicity, but of course we think of $\Chp$ as a Morita fibrant replacement of $k$ and one can show that $\cat Y(K)$ is weakly equivalent to $k^{K}$ as constructed in $\dgCat_{Mor}$, cf. the proof of Theorem \ref{thm-monodromy}.
\end{rk}

As we will mainly consider topological spaces via the functor $\Sing$ in this section we restrict attention to compactly generated Hausdorff spaces so that $\Sing$ is part of a Quillen equivalence.

\begin{lemma}\label{lemma-y-quillen}
The functor $K \mapsto \cat Y(K)$ is a left Quillen functor from $\sSet$ to $\dgCat_{DK}\op$ with right adjoint given by $\Map(-, \Chp)$,
\end{lemma}
\begin{proof} 
This follows for example from Theorems 16.4.2 and 16.5.7 in \cite{Hirschhorn03}.
\end{proof}

As all simplicial sets are cofibrant we obtain the following corollaries:
\begin{cor}\label{lemma-infinity-htpy}
The functor $K \mapsto \cat Y(K)$ preserves weak equivalences.
\end{cor}
\begin{cor}
The functor $K \mapsto \cat Y(K)$ sends homotopy colimits to homotopy limits.
\end{cor}
Since $\Sing$ sends cofibrations of topological spaces to cofibrations in $\sSet$ the lemma also holds for $\cat Y\colon \mathbf{CGHauss} \to \dgCat_{DK}\op$. Moreover, as $\Sing$ is a Quillen equivalence it preserves weak equivalences and homotopy colimits.
Then the last result can be interpreted as a Mayer--Vietoris theorem:
\[\cat Y(U \cup V) \simeq \cat Y(U) \times_{\cat Y(U \cap V)} \cat Y(V)\]

This definition of $\oo$-local systems looks a little indirect. But note that an $\oo$-local systems does provide us with an object of $(\Chp)_{n}$ for every $n$-simplex of $K$. One can consider an explicit simplicial resolution $(\Chp)_{*}$ as constructed in \cite{HolsteinC} to see this is the data one would expect.

Section \ref{sect-loop-rep}
will provide a more explicit way of looking at $\oo$-local systems, but first we show that $\oo$-local systems are equivalent to Morita cohomology.

Fix a topological space $X$ with a good hypercover $\{U_i\}_{i \in I}$. 

\begin{thm}\label{thm-monodromy}
The dg-categories $\cat H^{M}(X)$ and $\cat Y(X)$ are isomorphic in $Ho(\dgCat_{DK})$. 
\end{thm}
\begin{proof}
By Proposition \ref{propn-hypercover-colim} there is a weak equivalence \mbox{$\hocolim U_{*} \simeq X$}. Let $I = \cup I_{n}$ by the indexing category. Then we can consider the data of the category $I$ as a simplicial set $n \mapsto I_{n}$ with the induced face and degeneracy maps, or in fact as a simplicial space where every $I_{n}$ is considered as a discrete space.
Then we can consider the comparison map from $U_{n}$ to $\amalg_{I_{n}} *$ sending every connected open to a distinct point to get $\hocolim_{\De} U_{n} \simeq \hocolim_{\De} I_{n}$ where we take homotopy colimits of simplicial spaces. 
Then
 $I_{*}$ considered as a simplicial space has free degeneracies in the sense of Definition A.4 in \cite{Dugger04b}. Hence we can apply Theorem 1.2 of \cite{Dugger04b} and find $|I_{*}| \simeq \hocolim_{\De} I_{n}$. So the simplicial set $I_{*}$ is weakly equivalent to $\Sing X$ and it suffices to analyse $\Chp^{I_{*}}$. 

Hence by Remark \ref{rk-cotensor-holim} and Theorem \ref{thm-cohomology-holim} we are left to compare $\holim_{I\op} \Chp$ and $\holim_{\De I_{*}\op} (\Chp)_{n}$. But the category $I$ is exactly the category of simplices of the simplicial set $I_{*}$ and the weak equivalences $\Chp \to (\Chp)_{n}$ induce a weak equivalence of homotopy limits.
\end{proof}

\begin{rk}
This argument still applies if we replace $\Chp$ by any other dg-category $P$. Hence we know that $R\Ga(X, \underline P) \simeq P^{\Sing K}$. For example $R\Ga(X, \Chd) \simeq \cat Y^{u}(X)$. \end{rk}

\begin{cor}
The functor $X \mapsto \cat H^M(X)$ is homotopy invariant and sends homotopy colimits to homotopy limits.
\end{cor}
\begin{proof}
This is immediate from Theorem \ref{thm-monodromy} and the topological versions of Lemma \ref{lemma-y-quillen} and its corollaries.
\end{proof}
\begin{defn}
With this equivalence in mind we can define the \emph{Morita homology} $\cat H_M(K)$ of a simplicial set $K$ as $\Chdp \otimes K$. 
\end{defn}
Note, however, that computing this involves a cosimplicial resolution in dg-categories which looks difficult to produce. 

\subsection{Loop space representations}\label{sect-loop-rep}

In this section we move from the rather abstract action of simplicial sets on $\dgCat$ to representations of a dg-algebra.

For this we will have to move between $\dgCat$ and the category of linear simplicial categories. 
First recall that the natural smart truncation functor $\nneg$ from $\Ch$ to $\Ch_{\geq 0}$ extends to a functor from $\dgCat$ to $\dgCat_{\geq 0}$, which we also denote $\nneg$. This functor is right Quillen with left adjoint the inclusion functor.

Further recall the category $\sModCat_k$ of categories  enriched over simplicial \mbox{$k$-modules} and the natural Dold--Kan or Dold--Puppe functor $DK\colon \dgCat_{\geq 0} \to \sModCat$ that is defined hom-wise. $DK$ and its left adjoint $N$, normalization, form a Quillen equivalence between non-negatively graded dg-categories and $\sModCat$. 
For details see section 2.2 of \cite{Simpson05} or \cite{Tabuada10}. 

In V.5 of \cite{Goerss99} explicit looping and delooping functors for simplicial sets and simplicial groupoids are constructed. For arbitrary simplicial sets there is a functor $G\colon \sSet \to \mathbf{sGpd}$ with right adjoint $\overline W$. Together they form a Quillen equivalence.
The obvious composition with the normalization functor
$Nk G\colon \sSet \to \dgCat_{DK}$ is left Quillen. Essentially this lets us consider a simplicial set as a dg-category.
The restriction of $G$ to simplicial sets with a single vertex is a Quillen equivalence with simplicial groups.

Next we consider the enriched hom-space $R\uHom(\cat D, \Chd)$ of dg-categories.
As in the introduction, for any dg-category $\cat D$ we consider $L(\cat D\mods)$, where $L$ just restricts to the quasi-equivalent subcategory of fibrant and cofibrant objects. Let us write this as $\Chd^{\cat D}$. 
We note that this is quasi-equivalent to $R\uHom(\cat D, \Chd)$.  This is immediate from the definition if $\cat D$ is cofibrant. Otherwise consider a cofibrant replacement $j: \cat D \to Q\cat D$ and note that $\cat D\mods$ and $Q\cat D\mods$ are Quillen equivalent via $j^{*}$ by the results of Section 4.1 in \cite{Toen07a}. This shows that the comparison map of underlying dg-catgeories is quasi-essentially surjective. Moreover $j^{*}$ is compatible with shifts, so the equivalence of homotopy categories implies that the hom-spaces of $\cat D\mods$ and $Q\cat D\mods$ are quasi-isomorphic and this proves that $j^{*}$ is a quasi-equivalence.

Similarly write $\Chp^{\cat D}$ for the subcategory of $\Chd^{\cat D}$ consisting of objects whose underlying complexes are all perfect over $k$. This subcategory is preserved by $j^{*}$ and its adjoint so we have $\Chp^{\cat D} \simeq \Chp^{Q\cat D}$ and hence $\Chp^{\cat D} \simeq R\uHom(\cat D, \Chp)$.

\begin{thm}
For a simplicial set $K$ the dg-categories $\Chp^{K}$ and $
\Chp^{{NkGK}}$ are quasi-equivalent, as are $\Chd^{K}$ and $\Chd^{NkGK}$.
\end{thm}
\begin{proof}
The proofs for $\Chd^{(-)}$ and $\Chp^{(-)}$ are identical, so let us abusively write $\Ch$ for both.

By the Yoneda embedding it is enough to prove 
\[\Map_{\dgCat_{DK}}(\cat D, \Ch^{K}) \simeq \Map_{\dgCat_{DK}}(\cat D, \Ch^{NkGK})\]
for arbitrary dg-categories $\cat D$. (In fact an isomorphism of connected components of the mapping space would be enough.)

The left-hand side is $\Map_{\sSet}(K, \Map_{\dgCat}(\cat D, \Ch))$ by the usual adjunction. Meanwhile, for the right-hand side we have the following computation. We use the adjunctions $\otimes^{L} \dashv R\uHom$, $\iota \dashv \nneg$ (inclusion and truncation), $N \dashv DK$ (Dold--Kan), $k \dashv U$ (free and forgetful) and $G \dashv \overline W$ (looping and delooping). For legibility we contract $DK \circ \nneg$ to $DK$ and suppress $\iota$ \mbox{and $U$}. 
\begin{align*}
\Map(\cat D, \Ch^{NkGK}) 
& \simeq \Map_{\dgCat_{DK}}(NkGK, R\uHom(\cat D, \Ch)) \\ 
& \simeq \Map_{\dgCat_{DK}}(NkGK, \nneg(R\uHom(\cat D, \Ch))) & \textrm{as LHS  $\subset Im(\nneg)$}\\
& \simeq \Map_{\sModCat}(k GK, DK(R\uHom(\cat D, \Ch))) \\
& \simeq \Map_{\sCat}(GK, DK(R\uHom(\cat D, \Ch)))  \\
& \simeq \Map_{\mathbf{sGpd}}(GK, DK(R\uHom(\cat D, \Ch))) & \textrm{as LHS is a groupoid} \\
& \simeq \Map_{\sSet}(K, \overline W(DK(R\uHom(\cat D, \Ch))))
\end{align*}
Hence it suffices to show that $\overline W (DK ( R\uHom(\cat D, \Ch)))$ is weakly equivalent to $\Map(\cat D, \Ch) = \Map(\id, R\uHom(\cat D, \Ch))$. Since any simplicial set $K$ is weakly equivalent to $\Map(*, K)$ we consider the following.
\begin{align*}
\Map_{\sSet}( *, \overline W ( DK ( R\uHom(\cat D, \Ch)))) &\simeq \Map_{\mathbf{sGpd}}(*, DK(R\uHom(\cat D, \Ch))) \\
& \simeq \Map_{\sModCat}(\id, DK(R\uHom(\cat D, \Ch))) \\
& \simeq \Map_{\dgCat_{{DK}}}(\id, R\uHom(\cat D, \Ch))
\end{align*}
Here we 
use some of the same observations as before and note moreover that $G* \simeq *$, 
the trivial simplicial groupoid. Here the unit $\id$ is the one object category with morphism space $DK(k)$ respectively $k$.
\end{proof}

\begin{notation} If $X$ is a topological space we write $N\Om X$ for $N(k G \Sing(X))$.
\end{notation}

We can restrict from the dg-category $N(\Om X)$
to a more familiar dg-algebra if $X$ is connected and pointed. Let $\Om X$ denote the topological group of based Moore loops on $X$. Then $C_{*}(\Om X) \coloneqq C_{*}(\Om X, k)$ is a dg-algebra.
\begin{lemma}
Let $X$ be a pointed and connected topological space. $C_{*}(\Om X)$ considered as a dg-category with one object is quasi-equivalent to $N\Om X$.
\end{lemma}
\begin{proof}

$\Sing X$ is a connected simplicial set and by the existence of minimal Kan complexes  
has a reduced model $K$, i.e.\ there is a weakly equivalent simplicial set with a single vertex.

Then we have $G \Sing X \simeq GK$ as simplicial groupoids and thus as simplicial categories.
It follows that $N (k G \Sing X) \simeq N kGK$. 
Finally, there is a weak equivalence of simplicial groups between $GK$ and $\Sing \Om X$.
\end{proof}

Since quasi-equivalent dg-categories have quasi-equivalent categories of modules by our earlier discussion we have the following corollary.
\begin{thm}\label{thm-loopreps}
The dg-categories $\Chp^{C_{*}(\Om X)}$ and $\cat Y(X)$ are Morita equivalent, as are $\Chd^{{C_{*}(\Om X)}}$ and $\cat Y^{u}(X)$. 
\end{thm}
 
We can sum this up as a slogan: Morita cohomology is controlled by chains on the loop space. We will construct explicit models for $C_*(\Om X)$ in \mbox{Theorem \ref{thm-cellular}}. 

\begin{eg}
The category of loop space representations of $S^{2}$ is quasi-equivalent to the category of bounded chain complexes with a degree 1 endomorphism.
This follows since the homology algebra of $\Om S^2$ is equivalent to a polynomial algebra on a single generator in degree 1. See Section \ref{ch-eg} for more examples.
\end{eg}

\subsection{Cellular computations}\label{sect-cellular}
The previous computations
 correspond to computing \v Cech cohomology and singular cohomology of topological spaces. This is often not the most effective way of computing, and it becomes very cumbersome when we deal with coefficient categories.

In this section we will write down a simpler way of computing a model for $\Chp^{C_* (\Om X)}$ if $X$ is a CW-complex. This model will be given by representations of an algebra $\cat B(X)$ with a generator in degree $e-1$ for every $e$-cell (with an inverse if $e=1$). One could think of this as categorified cellular cohomology. The case for $\Chd^{C_* (\Om X)}$ works exactly in the same manner and for simplicity we write $\Ch^{(-)}$ for both cases again.

Note that if $X$ has no 1-cells and $k$ is a field one can construct $\cat B(X)$ as a cofibrant dg-algebra weakly equivalent to $C_{*}(\Om X)$.

For later reference we note:
\begin{lemma}\label{lemma-ch-colim}
The functor $\cat D \mapsto \Ch^{\cat D}$ sends colimits to limits.
\end{lemma}
\begin{proof}
The construction $\cat D \mapsto \cat D\mods$ is the naive category of dg-functors and is adjoint to the tensor product $- \otimes \Ch$. All objects are fibrant so we are left to compare cofibrants in $(\colim_{i}\cat A_{i})\mods$ with the limit of the categories of cofibrants in $\cat A_{i}\mods$. But since acyclic fibrations agree, the left lifting property gives the same conditions on both sides. 
\end{proof}

Next we compute an explicit model for $\Ch^{N\Om(X)}$. 
The plan is to proceed by induction on the cells of $X$. To perform this we first need good models for the cofibrations $N\Om S^{n-1} \hookrightarrow N\Om B^{n}$. 

Let $D(n)$ be the differential graded algebra $k[x_{n-1}, x_n \ | \ dx_n = x_{n-1}]$. Let $S(n) = k[x_{n} \ | \ dx_n = 0]$. Then  $k \to S(n)$ and $S(n-1) \to D(n)$ are the generating cofibrations for the model structure on dg-algebras. 

First we observe that ${S(n-1)} \simeq {N\Om S^{n}}$ if $n>1$. In other words $S(n-1)$ provides a model for singular chains on $\Om S^n$ equipped with the Pontryagin product. 
This is of course well-known, but one can also prove it directly using our set-up, see Example \ref{eg-oolocal-sn} in Section \ref{ch-eg}.

We also need to know that there is a map $D(n) \to N\Om B^{n}$ compatible with $S(n-1) \to D(n)$. This follows by the lifting property of the cofibration $S(n-1) \to D(n)$ with respect to the trivial fibration \mbox{$N\Om B^{n} \to *$}.

These are the building blocks needed to associate to any connected CW-complex $X$ 
a dg-algebra $\cat B(X)$ that approximates the way $X$ is glued from cells.

The following result already appears in \cite{Adams56a}. 

\begin{thm}\label{thm-cellular-2}
Associated to every connected CW complex $X$ with cells in dimension $\geq 2$ there is a cofibrant dg-algebra $\cat B(X)$ with one generator in degree $n-1$ for every $n$-cell, that is quasi-equivalent to $N(\Om X)$. 
\end{thm}

In particular $\cat Y(X) \simeq \Chp^{\cat B(X)}$ and $\cat Y^{u}(X) \simeq \Chd^{\cat B(X)}$. 
In the next theorem we will consider the case of 1-cells.

\begin{thm}\label{thm-cellular}
Associated to every connected CW complex $X$ there is a dg-algebra $\cat B(X)$ with one generator in degree $n-1$ for every $n$-cell with $n \geq 2$, and with two inverse generators in degree 0 for every 1-cell, such that $\cat Y(X) \simeq \Ch^{\cat B(X)}$.
\end{thm}
\begin{proof}Let us define $S^{*}(0) = k[a, a\inv]$ and $D^{*}(1) = k[a, a\inv, b \mapsto a-1]$ and consider the cofibration $S^{*}(0)  \hookrightarrow D^{*}(0)$. Of course $D^{*}(0) \simeq k$.

Then we have compatible quasi-isomorphisms $N\Om S^{1} \to S^{*}(0)$ and $N\Om B^{2} \to D^{*}(1)$. The first is induced by projection to connected components $G\Sing S^{1} \to \set Z$, the second map exists since $D^{*}(1) \to 0$ is a trivial fibration and $N\Om S^{1} \to N\Om B^{2}$ is a cofibration.

Let $X_{1}$ be the 1-skeleton of $X$ and  define 
$\cat B(X_{1}) = \cat B(\bigvee_{s} S^{1}) \coloneqq \otimes_{s} S^{*}(0)$ which is weakly equivalent to $C_{*}(\Om (\bigvee_{s} S^{1}))$
There is an obvious map from $S^{*}(0)$ to $\cat B(X_{1})$ for any attachment map $S^{1} \to X_{1}$. Assume first that $X$ is obtained from $X_{1}$ by attaching a 2-cell. Then we define 
\[ \cat B(X) = \colim \left(D^{*}(1) \leftarrow S^{*}(0) \to \cat B(X_{1})\right)\]

Now $\cat Y(X)$ is the homotopy pullback of $\cat Y(B^{2}) \leftarrow \cat Y(S^{1})\to \cat Y(X_{1})$. But this diagram is weakly equivalent to
$\Ch^{N\Om B^{2}} \to \Ch^{N\Om S^{1}} \leftarrow \Ch^{N\Om X_{1}}$
which is in turn weakly equivalent to
$\Ch^{D^{*}(1)} \to \Ch^{S^{*}(0)} \leftarrow \Ch^{\cat B(X_{1})}$.

These are all pullback diagrams of fibrant objects with one map a fibration, hence they are homotopy pullbacks as $\dgCat_{DK}$ is right proper since every object is fibrant. Since the diagrams are levelwise quasi-equivalent their pullbacks are quasi-equivalent, and thus also isomorphic in $Ho(\dgCat_{Mor})$. But since $\cat D \mapsto \Ch^{\cat D}$ sends colimits to limits by Lemma \ref{lemma-ch-colim} it also follows that 
\begin{align*}
\cat Y(X) 
&\simeq \holim \left( \cat Y(B^{2}) \to \cat Y(S^{1}) \leftarrow \cat Y(X_{1})\right) \\
&\simeq \holim \left(\Ch^{D^{*}(1)} \to \Ch^{S^{*}(0)} \leftarrow \Ch^{\cat B(X_{1})}\right)\\
&\simeq \lim \left( \Ch^{D^{*}(1)} \to \Ch^{S^{*}(0)} \leftarrow \Ch^{\cat B(X_{1})} \right) \\
&\simeq \Ch^{\colim \left(D^{*}(1) \leftarrow S^{*}(0) \to \cat B(X_{1})\right)}
\end{align*}
The colimit in the exponent is how we have defined $\cat B(X)$.

Now consider the general case. First to obtain $\cat B(X_{2})$ note that any attachment map from $S^{1}$ factors through $X_{1}$, so we can repeat the previous step as often as required.
Attachment of higher-dimensional cells works in exactly the same manner, we just have to replace $S^{*}(0)$ by $S(n-1)$ and $D^{*}(1)$ by $D(n)$.

To extend to infinite CW-complexes we have to check the same argument goes through for filtered colimits. Since the maps $X_{< \al} \to X_{\leq \al}$ are cofibrations the filtered colimit is a homotopy colimit and commutes with $N\Om$. So $N\Om X_{\leq \la} \simeq \hocolim_{\al < \la} N\Om X_{\al}$ and we can define $\cat B(X_{\leq \la})$ as $\colim_{\al < \la} \cat B(X_{\leq \al})$.
\end{proof}

\begin{rk}
To use this computation in practice 
we need to identify the degree $n-1$ element $y$ of $\cat B(X_{< \al})$ that corresponds to the image of $S^{n-1}$. Then we adjoin a new generator $x$ with $dx = y$. This can of course be quite non-trivial. There are some examples in Section \ref{ch-eg}.
\end{rk}

\begin{rk}
By construction $\cat B(X)$ is Morita-equivalent to $N\Om X$, but it does not follow from the construction whether the two dg-algebras are isomorphic in $Ho(\dgAlg)$.
\end{rk}

\subsection{Finiteness and Hochschild homology}\label{sect-hochschild}

In this section we consider conditions for Morita cohomology to satisfy some finiteness properties, and determine Hochschild (co)homology in several cases by quoting relevant results from the literature.

Let us first make some definitions. Here $R$ denotes fibrant replacement in $\dgCat_{Mor}$. Specifically, $RB = L(B\op \mods)_{pe}$. 

We say a dg-category $\cat D$ is \emph{locally proper} if the hom-space between any two objects is a perfect complex. $\cat D$ is \emph{proper} if moreover the triangulated category $H_{0}(R\cat D)$ has a compact generator, i.e.\ a compact object which detects all objects. 

Recall an object $X$ in a model category is \emph{homotopically finitely presented} if $\Map(X, -)$ commutes with filtered colimits. 
$\cat D$ is \emph{smooth} if it is homotopically finitely presented as a $\cat D\op \otimes \cat D$-module.
$\cat D$ is \emph{saturated} if it is smooth, proper and Morita fibrant. 

We say $\cat D$ is \emph{of finite type} if there is a homotopically finitely presented dg-algebra $B$ such that $R\cat D \simeq R(B\op)$.

These definitions are Morita-invariant (except for the condition of being Morita fibrant). 
To\"en shows in Lemma 2.6 of \cite{Toen07} that a dg-category has a compact generator if and only if $R\cat D \simeq RB\op$ for some dg-algebra $B$ and is moreover proper if and only if the underlying complex of $B$ is perfect. Moreover any dg-category of finite type is smooth (Proposition 2.14 of \cite{Toen07}).

\begin{rk}
Saturated dg-categories are precisely the dualizable objects in $Ho(\dgCat_{Mor})$.
Another reason to be interested in this finiteness condition is that if a dg-category is saturated there is a nice moduli stack of objects, this is the main result of \cite{Toen07}.
\end{rk}

\begin{propn} The dg-category $\cat Y^{u}(X)$ is triangulated and has a compact generator. If $X$ is a finite CW-complex without 1-cells then $\cat Y^u$ is smooth. If moreover $H_{*}(\Om X)$ is of finite type then $\cat Y^{u}(X)$ is saturated.
\end{propn}
\begin{proof}
Note first that as a homotopy limit $\cat Y^{u}(X)$ is fibrant and the compact generator is given by $C_{*}(\Om X)$. 

Theorem \ref{thm-cellular}
implies that in the absence of 1-cells the dg-algebra $\cat B(X)$ is homotopically finitely presented. 
So the category $\cat Y^{u}(X)$ is of finite type and hence smooth.
If $H_{*}(\Om X)$ is of finite type, then $\cat B(X)$ is a perfect complex over $k$, and $\cat Y^{u}$ is moreover proper and we find that $\cat Y^{u}(X)$ is saturated. 
\end{proof}

In fact Kontsevich shows in \cite{Kontsevich09} that the dg-algebra of chains on the loop space of a finite connected CW complex is always of finite type.

By contrast if $X$ is an infinite CW-complex then $\cat B(X)$ is usually not homotopically finitely presented.
For example consider $\cat B(\C P^{\oo}) \simeq k[x_{1}]/(x_{1}^{2})$ where $x_{1}$ is in degree 1. 

Next we consider properness for $\cat Y(X)$.
The category $\cat H^{M}(X)$ is locally proper if all cohomology groups of $X$ with coefficients in local systems are finite dimensional and concentrated in finitely many degree. 
This is for example the case if $X$ has a finite good cover. Then the hom-spaces are finite limits of perfect chain complexes. 

This is in contrast to Ext-groups of local systems, which can be large even if $X$ is very well behaved, for example if $X$ is a smooth projective variety \cite{Dimca09}.

The example $X = S^{1}$ shows that we cannot expect $\cat Y(X)$ to be proper in general. $\Ch^{S^{1}}$ is the category of complexes of $\Z$-representations, with infinitely many connected components. 

\begin{propn}
If $\pi_{1}(X)$ has only finitely many irreducible finite-dimensional representations then there exists a compact generator $A$ and $\cat Y(X) \simeq L(\End(A)\op\mods)_{{pe}}$. 
Moreover, $\cat Y(X)$ is proper if $C^{*}(X, \End(A))$ is a perfect complex.
\end{propn}
\begin{proof}
We define $A$ to be the sum of all the irreducibles. This clearly generates the dg-category. 
By Lemma 2.6 of \cite{Toen07} $L(\cat Y(X)\op\mods) \simeq L(\End_{\cat Y(X)}(A)\op\mods)$.
Since $\cat Y(X) \simeq L(\cat Y(X)\op\mods)_{pe}$ we deduce that $\cat Y(X)$ is the subcategory of compact objects in $\End(A)\mods$.

The second statement is clear.
\end{proof}

The proposition applies for example if the fundamental group is finite. Then we can take $A$ to be the group ring. 

\begin{eg}\label{eg-generator}
Let $X$ be simply connected. Then we can take $A = k$ and find $\End(A) \simeq R\uHom_{\Om X}(k, k) \simeq C^{*}(X, k)$. The second quasi-isomorphisms follows for example from results in \cite{HolsteinB}. 
 In particular $\cat Y(X) \simeq C^{*}(X,k)$ in $\dgCat_{Mor}$. 
Then $\cat Y(X)$ is proper if and only if $C^{*}(X, k)$ is a perfect complex. If $C^{*}(X, k)$ is homotopically finitely presented then $\cat Y(X)$ is moreover smooth and saturated. 
\end{eg}

If $\cat Y(X)$ has a compact generator it becomes much easier to compute secondary invariants. In particular we can compute Hochschild homology and cohomology. For definitions and background see \cite{Keller06}. Since Hochschild homology and cohomology are Morita-invariant we can compute them on a generator of a dg-category if there is one. 

Example \ref{eg-generator} implies the following proposition. Here $HH$ stands for either $HH^{*}$ or $HH_{*}$.
\begin{propn}
Let $X$ be simply connected, then $HH(\cat Y(X)) \cong HH(C^{*}(X))$.
\end{propn}
So we can compute Hochschild (co)homology of Morita cohomology from minimal models (in the sense of Sullivan).

\begin{propn}
For any space $X$ there are isomorphisms $HH(\cat Y^{u}(X)) \cong HH(C_{*}(\Om X)) \cong HH(\cat B(X))$. 
\end{propn}
\begin{proof}
The first isomorphisms follows from Corollary 8.2 in \cite{Toen06},  the second isomorphism follows since Hochschild (co)homology is Morita-invariant.
\end{proof}

The following applications follows from results readily available in the literature.

\begin{propn}
Let $X$ be simply connected then $HH_{*}(\cat Y(X)) \cong H^{*}(\cat L X)$. 
If $M$ is a simply connected closed oriented manifold of dimension $d$ then $HH^{*}(\cat Y(M)) \cong H_{*+d}(\cat L M)$ as graded algebras with the Chas-Sullivan product on the right hand side.
\end{propn}
\begin{proof}
If $X$ is simply connected it is well known (see \cite{Loday11}) 
that $HH_{*}(C^{*}(X, k)) \cong H^{*}(\cat L X)$ where $\cat L X$ is the free loop space. 

The second part follows since the Hochschild cohomology ring of singular cochains on $M$ (with the cup product) is isomorphic to its loop homology with the Chas-Sullivan product, 
cf. \cite{Cohen02}.  
\end{proof}

Note that we do not expect Hochschild homology of $\cat Y(X)$ to be particularly tractable if $X$ is not simply connected. For example $\cat Y(S^{1})$ is equal to dg-representations of $\set Z$ and has $|k^{*}|$ simple objects with no morphisms between them. Hence it follows from the explicit definition in \cite{Keller06} that Hochschild homology consists of $|k^{*}|$ copies of $HH_{*}(k[y])$ where $y$ lives in degree 1 and has square $0$.
\begin{propn}
For any space $HH_{*}(\cat Y^{u}(X)) \cong H_{*}(\cat L X)$. 
If $X$ is a simply connected CW complex there is an isomorphism of graded algebras $HH^{*}(\cat Y^{u}(X)) \cong H^{*}(\cat L X)$.
\end{propn}
\begin{proof}
We find $HH_*(\cat Y^u(X)) \cong HH_* \Om(X) \cong H_*(\cat LX)$ from 7.3.14 in \cite{Loday92}. 

The result that $HH^*\Sing \Om X \cong H^*(\cat L X)$ as graded algebras if $X$ is simply connected is in \cite{Menichi01}.
\end{proof}

\section{Examples}\label{ch-eg}

In this section we compute some examples of Morita cohomology. We will mainly use the characterization in terms of $C_{*}(\Om X)$ or the dg-algebra $\cat B(X)$ defined in Section \ref{sect-cellular}.

In the following whenever an element has a subscript, this will denote its degree.

\begin{eg}\label{eg-s1}
We begin with the case $X = S^{1}$. Clearly $\cat H^M(S^1)$ is equivalent to the category of representations of $\Z \simeq \Om S^1$.

This is also the category of bounded chain complexes of local systems on $S^1$. 

We can also characterize $\cat H^M(S^1)$ as the explicit homotopy limit 
\[(\Chp)^I \times^{h}_{\Chp \times \Chp} \Chp\]
Here $\Chp^I$ is the path object in dg-categories, see for example \cite{HolsteinC}. 
The limit then comes out as the category of pairs $(M, \phi \in \Aut(M))$ with morphisms $(f,g,h)\colon  (M, \phi) \to (N, \psi)$ in $\uHom(M,N)^{\oplus 2} \oplus \uHom(M,N)[-1]$ with differential
\[(f,g,h) \mapsto (df, dg, dh -(-1)^{|g|} g \phi + \psi f)\]
In particular $\uHom^*(k, k) \cong k \oplus k[1]$, which is exactly cohomology of $S^1$, as predicted.

Note that the category $\cat H^M(S^1)$ is highly disconnected, in fact isomorphism classes of simple objects are naturally in bijection with $k^*$. Of course $k^*$ has a geometric structure, and one way of interpreting large sets of isomorphism classes of objects is to consider a moduli stack of objects of $\cat H^M(X)$. We will not follow this direction here.
\end{eg}

\begin{eg}\label{eg-oolocal-sn}
If $n>1$ then $\cat H^{M}(S^{n}) \simeq \Chp^{S(n)}$, i.e.\ the category of perfect chain complexes with an endomorphism in degree $n-1$. 
\end{eg}
\begin{proof}[Proof 1]
This is a consequence of the quasi-\-isomorphism $S(n) \to N(\Om \Sing S^n)$ which follows form the well-known computation of $H_{*}(\Om \Sing S^{n})$. 
\end{proof}
\begin{proof}[Proof 2]
We can also compute $\cat B(S^{2})$ using the method of Theorem \ref{thm-cellular} by gluing two copies of $B^{2}$ along $S^{1}$. The resulting dg-algebra has one invertible generator with two trivialising homotopies, which is quasi-isomorphic to \mbox{$k[x_{1}] = S(1)$}.

Once we know the case $n=2$ we can inductively compute $S^{n} = D^{n} \amalg_{S^{n-1}} D^{n}$ and note that $S(n) \simeq D(n) \otimes^{L}_{S(n)} D(n)$. 

Note that we can use this construction of $\cat B(S^n)$ in the proof of Theorem \ref{thm-cellular}.
There is no circularity as we only need a model for spheres in dimension less than $n$ to compute $\cat B(S^n)$.
\end{proof}

\begin{eg} Next consider some more detail for $n=2$.
Since $k$ is a generator let $A \coloneqq R\underline \End_{C_{*}(\Om S^{2})} \simeq C^{*}(S^{2}) \simeq k[x_{2}, x_{3} \stackrel d \to x_{2}^{2}]$ and we can characterize $\cat Y(S^{2})$ as compact objects in $A\mods$.

An example of an object of $R\Ga_{Morita}(S^2, k)$ is the chain complex associated to the Hopf fibration $p\colon  S^3 \to S^2$. As a homotopy locally constant sheaf we can consider this as $Rp_* \Sing(S^3)$. As a representation of $\Om S^2$ this can be written as $k \oplus k[-1]$ with the natural self-map of degree 1.

Since $\pi_1(S^2)$ is trivial, we can also view $R\Ga_{Mor}(S^2, \underline k)$ as generated by the trivial local system and the information $H^*(S^2, -)$ provides about (iterated) extensions. This provides a slightly different viewpoint on Morita cohomology.

\end{eg}

\begin{eg}
For a group $G$ it is clear that $\cat H^{M}(BG)$ is just the dg-category of perfect complexes with an action of $G$.
\end{eg}

\begin{eg}
The dg-category $\cat H^{M}(\R P^2)$ is given by representations of $\cat B(\R P^2)$ on perfect complexes, and $\cat B(\R P^2)$ has generators $a_{0}, a_{0}\inv, b_{1}$ such that $db_{1} = a_{0}\circ a_{0}-1$. 
The identification $db_{1} = a_{0} \circ a_{0} -1$ is induced by the attaching map from the boundary of the 2-cell to $\R P^{1}$. 
rationally good.

We can obtain $\cat B(\R P^{3})$ from $\cat B(\R P^{2})$ by adding $c_{2}$ with $dc_{2} = 0$.

If we are working over the field $\Q$ Morita cohomology has certain similarities to rational homotopy theory, cf. the duality between $C_{*}(\Om X)$ and $C^{*}(X)$ in the simply connected case. 
On the other hand we see that $\R P^{2}$ has trivial minimal model, but its Morita cohomology is a dg-category with two simple objects corresponding to the irreducible representations of $\Z/2$.

\end{eg}

\begin{eg}
Next we compute the map $p^*\colon  \cat H^{M}(S^{2}) \to \cat H^{M}(S^{3})$ induced by the Hopf fibration.

On the level of loop spaces we see that the map is induced by $\Om p_*\colon  H_*(\Om S^3) \to H_*(\Om S^2)$ which is given by $x_2 \mapsto y_1^2$ on the generators.

With this in mind we can work out $\cat H^{M}(\C P^2)$ explicitly by considering the following diagram: 
\[\cat H^{M}(B^4) \stackrel {i^*} \longrightarrow \cat H^{M}(S^3) \stackrel {p^{*}}\longleftarrow \cat H^{M}(\C P^1)\]

On the level of dg-algebras we have
\[D(3) \stackrel {i_{*}} \longleftarrow S(3) \stackrel {p_{*}} \longrightarrow \cat B(S^{2}) \cong S(2)\]

The attaching map $p_{*}$, is induced by the Hopf fibration. As we have just seen 
it corresponds to the map $H_{*}(\Om S^{3}) \to H_{{*}}(\Om S^{2})$ given by sending $x_{2} \mapsto y_{1}^{2}$. Hence we find:
\[\cat B(\C P^{2}) \simeq k[\al_{1}, \al_{3} \ | \ d\al_{3} = \al_{1}^{2}]\]
\end{eg}

\begin{eg}
We can generalise this to $\C P^n$, every extension over a $2i$-cell corresponding to another map $\al_{2i-1}$ in degree $2i-1$. We find $d\colon  \al_3 \mapsto \al_1^2$; \ $\al_5 \mapsto \al_3 \al_1 + \al_1 \al_3$; \ $\al_7 \mapsto \al_5 \al_1 + \al _3^2 + \al_1 \al_5$ etc. 

It is well-known that $H_*(\Om \C P^n)$ is isomorphic to \mbox{$\Lambda(y_1) \otimes k[y_{2n}]$} as a Hopf algebra, in particular the Pontryagin products agree. 
To relate this to the above description identify $y_{2n} = \al_{2n-1} \al_1 + \dots + \al_1 \al_{2n-1}$. The dg-algebra $\cat B(X)$ is larger since it is quasi-free (i.e.\ the underlying graded associative algebra is free), while $H_{*}(\Om \C P^{n})$ is only quasi-free as a commutative dg-algebra.
\end{eg}
\begin{eg}
Taking the limit we find $\cat B(\C P^\oo)$. 
Of course the homology algebra of $\Om \C P^\oo$ is just that of $S^1$. Indeed $k[\al_1, \al_3, \dots]$ with its differentials is a quasi-free model for $k[z_1]$. 
\end{eg}

We conclude with the following example of a space with trivial Morita cohomology.

\begin{eg}
Consider Higman's 4-group $H$ with the following presentation:
$$\langle a, b, c, d \ | \ a^{-1}ba=b^2,\ b^{-1}cb=c^2,\ c^{-1}dc=d^2,\ d^{-1}ad=a^2\rangle $$
This is an acyclic group without non-trivial finite dimensional representations. Its classifying space $BH$ is known to be a finite CW complex.
For references see e.g.\ \cite{Berrick02}.
It is easy to see that the Morita cohomology of $BH$ is quasi-equivalent to $\Chp$.
\end{eg}

\bibliography{../biblibrary}

\end{document}